\documentclass[11pt,reqno]{amsart}

\usepackage{amsmath,amssymb,amsfonts,enumerate,amsthm,graphicx}

\input xy
\xyoption{all}
 \newtheorem{thm}{Theorem}[section]
 
 \newtheorem{lem}[thm]{Lemma}
 \newtheorem{prop}[thm]{Proposition}
 \theoremstyle{definition}
 
 \newtheorem{example}[thm]{Example}
 \theoremstyle{remark}
 \newtheorem{rem}[thm]{Remark}

\DeclareMathOperator{\Image}{Im}

 \DeclareMathOperator{\beg}{indeg}
\DeclareMathOperator{\Deg}{deg} 

 \DeclareMathOperator{\Dim}{dim}

 \DeclareMathOperator{\Ht}{ht}
\DeclareMathOperator{\Reg}{reg} 
\DeclareMathOperator{\grade}{grade}
\DeclareMathOperator{\End}{end} \DeclareMathOperator{\END}{end}

 \def\ff{{\bf f}}

 \def\xx{{\bf x}}
 \def\yy{{\bf y}}

  \def\XX{{\bf X}}
  \def\YY{{\bf Y}}

\def\rar{\rightarrow}
\def\lar{\longrightarrow}

\def\pp{\mathbb{P}}
\def\ii{\'{\i}}

\newcommand{\fm}{\mathfrak{m}}
\newcommand{\fn}{\mathfrak{n}}

\newcommand{\fp}{\frak{p}}

\newcommand{\fa}{\frak{a}}
\newcommand{\fb}{\frak{b}}

\begin{document}

\title[Birational Maps with arithmetically Cohen-Macaulay Graphs ]
 { Bounds on degrees of Birational Maps with arithmetically Cohen-Macaulay Graphs}
\author[Seyed Hamid Hassanzadeh]{S. H. Hassanzadeh}
\address{ Universidade Federal de Rio de Janeiro
Departamento de Matem\'{a}tica, Cidade Universit\'{a}ria, Rio de Janeiro, Brazil. }
\email{hamid@im.ufrj.br}

\author[Aron Simis]{A. Simis}

\address{ Universidade Federal de Pernambuco
Departamento de Matem\'{a}tica, CCEN,
50740-560, Recife, Pernambuco, Brazil}
\email{aron@dmat.ufpe.br}

\footnotetext{Mathematics Subject Classification 2010
 (MSC2010). 13D02, 13D45, 13H10, 14E05, 14E07.}
\keywords{Graph of a rational map,  birational maps, Cohen--Macaulay Rees algebra, reduction number, regularity}

\begin{abstract}
A rational map whose source and image are projectively embedded varieties is said to have an {\em arithmetically Cohen--Macaulay graph} if the Rees algebra of one (hence any) of its base ideals is a Cohen--Macaulay ring. The main objective of this paper is to obtain an upper bound for the degree of a representative of the map in case it is birational onto the image with an arithmetically Cohen--Macaulay graph. In the plane case a  complete classification is given of the Cremona maps with arithmetically Cohen--Macaulay graph, while in arbitrary dimension $n$ it is shown that a Cremona map with arithmetically Cohen--Macaulay graph has degree at most $n^2$.
\end{abstract}

\maketitle

\section*{Introduction}

The overall goal of this work is to establish upper bounds for the degrees of representatives of a birational map.
The guiding idea is to obtain bounds that are simply expressed in terms of known numerical invariants, such as dimension, relation type, regularity and so forth.

An intriguing related challenge is to bound the degrees of representatives of the inverse map in terms of invariants of the originallly given birational map.
It is classical knowledge  that, over an algebrically closed field $k$,  the inverse map to a plane Cremona map $\mathfrak{F}:\mathbb{P}^2_k\dasharrow \mathbb{P}^2_k$ defined by forms  of  degree $d$  is also defined by forms of degree $d$ -- which is  essentially  an application of Bezout's theorem (see, e.g., \cite{alberich}). In arbitrary dimension $n$, the problem of determining a universal upper bound for  the degree of the inverse of a Cremona map of $\mathbb{P}^n$ was settled by  O. Gabber  (\cite{BCW}) who showed that if $\mathfrak{F}:\mathbb{P}^n_k\dasharrow \mathbb{P}^n_k$ is a Cremona map defined by forms  of  degree $d$ then its inverse is  defined by forms of degree at most $d^{n-1}$. This upper  bound is sharp as the map  defined by 
$$(x_0^d:x_1x_0^{d-1}:x_2x_0^{d-1}+x_1^d:x_3x_0^{d-1}+x_2^d:\cdots:x_nx_0^{d-1}+x_{n-1}^d)$$
can be shown to be a Cremona map with inverse map defined in degree $d^{n-1}$.

Enlarging the picture from Cremona maps to arbitrary birational maps $\mathfrak{F}:X\dasharrow  Y$ where $X\subset \mathbb{P}_k^n$ and  $Y\subset \mathbb{P}_k^m$, one faces preliminarily the fact that both the map and its inverse can be defined (represented) by different sets of forms of fixed degree, where this degree may vary from set to set.
One way to circumvent this technicality is to think about the smallest such degree -- in fact, it follows from \cite[Proposition 1.1]{S} that there are natural minimum and {\it maximum} such degrees for any rational map $\mathfrak{F}:X\dasharrow  Y$, while from Theorem~\ref{Tcriterion} these degrees in the case of the inverse of a given birational map $\mathfrak{F}:X\dasharrow  Y$ are in principle bounded in terms of the invariants of $\mathfrak{F}$.  

It is sufficiently evident that the problem of finding an upper bound for the possible degrees of the inverse map becomes very subtle. As far as the authors know there is no known general upper bound for arbitrary $X$ and $Y$. 
Yet, based on a recent bound for the degrees of reduced Groebner bases established in \cite{MR}, we will state one such bound.
As expected from Groebner basis theory, the resulting bound is doubly exponential.
We thank the referee for pointing out that such a bound possibly admits further improving with methods from elimination theory.
Although not pursued in this work, the question, beyond its own interest, stands out in a variety of applications, from the Jacobian conjecture to more recent work such as \cite{BBV}.

The more specific goal of this paper is to study how the Cohen--Macaulay property of the graph of the map has any impact on these bounds.
Under this requirement a much better bound for the minimum defining degree of the inverse map to a birational map is obtained in terms of the given data (Theorem~\ref{Tbirational}).
Here we think of the graph of the map in terms of the Rees algebra $\mathcal{R}_A(I)$ of a base ideal $I\subset A$ of the map, where $A$ denotes the homogeneous coordinate ring of the source projective variety.
The algebra $\mathcal{R}_A(I)$ is sufficiently stable to warrant a bona fide replacement of the graph - this and other aspects of the theory will be explained in the later sections.
The property that $\mathcal{R}_A(I)$ be a Cohen--Macaulay ring -- i.e., that the bi-projective structure of the graph be arithmetically Cohen--Macaulay -- is stronger than the property that the graph be a Cohen--Macaulay variety since the latter is a local property at geometric points off the bi-irrelevant ideal.
However, for the sake of brevity we will often refer to a Cohen--Macaulay graph when the stronger property holds.

We now describe the contents of the sections.

The first section is a recap of the main terminology and tools, with a pointer to the aspects that will play a central role in the statements and proofs of subsequent sections.
The first subsection is a recap of homological gadgets, while
the second discusses shortly the notion of a base ideal of a rational map and its associated algebraic constructions.
We restate the main effective birationality criterion of \cite{DHS} and as an application we give a general bound for the possible degrees of the inverse map to an arbitrary birational map.

The second section contains the main results.
It is further subdivided in four subsections.

The first subsection is a direct attack on the question of obtaining an exact bound for the defining degree of a Cremona map for which the corresponding graph is arithmetically Cohen--Macaulay.
The result only covers the plane case and shows that this bound is $4$ -- a particular case of a more general result to be subsequently proved.
In addition, one obtains a complete classification of the plane Cremona maps for which the corresponding graph is arithmetically Cohen--Macaulay.

 The second and third subsections are concerned with the main results of the paper, stating general bounds, first for the minimum defining degree of a birational map onto the image, and  second for the minimum defining degree of a birational map onto the image when the graph is in addition Cohen--Macaulay.
 The latter bound has a ``symmetric'' nature, thus also providing a bound for the minimum defining degree of the inverse map in terms of the given data.
Both bounds seem to be new in the literature.

We also discuss a related question on the asymptotic behavior of the regularity of the powers of the base ideal of a birational map.

We then state a crucial lemma which encapsulates a couple of estimates around numerical homological invariants of an ideal of a local ring.
This lemma is applied  to derive some stronger bounds in the case of a Cremona map in arbitrary dimension whose graph is Cohen--Macaulay.
In addition one proves that a plane Cremona map having Cohen--Macaulay graph has a saturated base ideal.
Again such results seem to be new.

The last subsection contains some ramblings on bounding the defining degrees of a rational map having a base ideal $I$ of grade at least $2$ -- a hypothesis saying that the map has essentially a unique base ideal. The statement gives an upper bound for the degree of the generators of $I$ by invoking its syzygies.
One proceeds to a comparison between this and the earlier bounds.

\section{Setup and terminology}

\subsection{Homological recap}\label{homological_recap}

Let $(A,\fm)$ denote a Noetherian local ring of dimension $d$ and let  $I$ stand for a proper ideal of $A$. Set $\mathcal{R}_A(I):=A[It]\subset A[t]$, the Rees algebra of $I$ over $A$ and $\mathcal{G}_I(A):=A/I\otimes_A\mathcal{R}_A(I)$, the associated graded algebra. Let $S=A[y_1,\cdots,y_r]$ denote a standard graded $^*$local polynomial ring over $A$ and let $\mathcal{J}$ stand for a presentation ideal of $\mathcal{R}_A(I)$ over $S$. Clearly, then $\mathcal{R}_A(I)\simeq S/\mathcal{J}$ is standard graded $^*$local, with homogeneous maximal ideal  $\mathfrak{N}:=(\fm,(\mathcal{R}_A(I))_+)$. The ideal $\mathcal{J}$ is a homogeneous ideal of $S$. The maximum degree of the elements of a minimal generating set of $\mathcal{J}$ is called the relation type of $I$. The relation type depends only on the ideal $I$ and is independent of its generators.
Similarly, the residue algebra $\mathcal{G}_I(A)$ is a *local ring with *maximal ideal $\mathfrak{M}:=(\fm,\mathcal{G}_I(A)_+)$.

We will be interested in the case where $\mathcal{R}_A(I)$ is a Cohen--Macaulay ring.
For this, recall more generally that if $S$ is a finitely generated $\mathbb{N}$-graded algebra over a local ring $(A,\fm)$ and its maximal ideal or a standard graded algebra over a field and its irrelevant ideal, with $S_0=A$, then $S$ is a Cohen--Macaulay ring if and only if the local ring $T:=S_{(\fm, S_+)}$ is Cohen--Macaulay (\cite{HR} or \cite[Exercise 2.1.17]{BH}).


An ideal $J\subset I$ is called a {}reduction of $I$ if $JI^n=I^{n+1}$ for some integer $n$. The reduction $J$ is said to be a {\em minimal reduction} if it does not contain properly another reduction of $I$.  The {\em reduction number} of $I$ with respect to the minimal reduction $J$, denoted by $r_J(I)$, is the least integer $n\geq 0$ such that $JI^n=I^{n+1}$. The (absolute) reduction number of $I$ is $\min\{r_J(I): J\, \text{is a minimal reduction of } I\}$. The {\em analytic spread} of $I$ is  $\ell(I):=\dim\,\mathcal{R}_A(I)/\fm \mathcal{R}_A(I)$. If the residue field $A/\fm$ is infinite then the analytic spread coincides with the minimum number of the generators of a minimal reduction of $I$.

 Let $R=\bigoplus _{n \geq 0} R_{n}$  be a
positively graded *local Noetherian ring of dimension $d$
with graded maximal ideal $\fn$. For a graded $R-$module $M$, $\beg(M):=\inf\{i : M_i\neq 0\}$ and
$\End(M):=\sup\{i : M_i\neq 0\}$.

 If $M$ is a finitely generated graded $R$-module, the
 Castelnuovo-Mumford regularity  of $M$ is defined as
  $$\Reg (M) := \max  \{\END  (H^i_{R_+}(M))+i\}$$ and the $a$-invariants of $M$ are defined as $a_i(M)=\END(H^i_{R_+}(M))$ as well as $a(M):=a_{\Dim(M)}$.  

One can also introduce the regularity
with respect to the maximal ideal $\mathfrak{M}$, namely, $\Reg _{\frak{M}}(M) := \max \{\END (H^i_{\frak{M}}(M))+i\}.$ Likewise, set  $a_i^*(M):=\END (H^i_{\mathfrak{M}}(M))$.

In particular, if $\dim A=d$ then
 $a^*_d(\mathcal{G}_I(A))=\END (H^d_{\frak{M}}(\mathcal{G}_I(A)))$, an invariant of relevance in this work.

 Finally, if $A$ is standard graded and $I\subset A$ is an ideal generated by forms of the same degree, then one can also consider the $(\xx)$-{\em regularity} by defining it to be $\Reg_{\xx}(\mathcal{R}_{A}(I):=\max\{\End(H^i_{(\xx)}(\mathcal{R}_{A}(I)))+i\}$, where the grading has been reseted to  $\deg(\xx)=1$ and $\deg(\yy)=0$.
 This notion will be frequently used throughout.

 The above four notions -- Cohen--Macaulayness, reduction number, regularity and $a$-invariant -- are inextricably tied together and there is quite a large publication list on the subject by several authors going back at least to the early nineties.
 Our purpose is to draw on parts of this knowledge in order to prove one of the main theorems in the paper.

\subsection{Elements of birational maps}\label{Birational_Maps}

Let $\mathfrak{F}:X\dasharrow Y$ denote a birational map of  (reduced and irreducible) projective varieties over an algebraically closed field $k$.
For simplicity, assume that $X\subset \mathbb{P}^n_k$ and $Y\subset \mathbb{P}^m_k$ are non-degenerate embeddings.
Let $A:=S(X)$ and $B:=S(Y)$ stand for the corresponding homogeneous coordinate rings. Set $d=\dim A=\dim B$.

Now, quite generally,  a rational map $\mathfrak{F}: X\dasharrow \mathbb{P}^m_k$ is defined by forms $f_0,\ldots,f_m \in A$ of the same degree -- the tuple $(f_0,\ldots ,f_m)$ is called a {\em representative} of the map and the $f_j$'s are its coordinates.
If no confusion arises, we call the common degree of the $f_j$'s the {\em degree of the representative}.
Set $I:=(f_0,\ldots,f_m)\subset A$ for the ideal generated by the coordinates of the given representative, referred to as a {\em base ideal} of $\mathfrak{F}$.
One may stress the fact that the base locus of the map -- in the sense of the exact subscheme where the map is not defined -- is given by another ideal containing every base ideal (see \cite[Proposition 1.1 (2)]{S}).

Now, the set of representatives of $\mathfrak{F}$ correspond bijectively to the
homogeneous vectors in the rank one graded $A$-module $\mbox{\rm Hom}(I,A)$.
This fact, and its consequence that  $\mathfrak{F}$ is uniquely represented up to proportionality if and only if $I$ has grade at least $2$, are explained in \cite[Proposition 1.1 and Definition 1.2]{S}.
We emphasize that two representatives of $\mathfrak{F}$ share the same module of syzygies, generate isomorphic $k$-subalgebras of $A$ and the corresponding base ideals of $A$ have naturally isomorphic Rees algebras (see \cite{S} and \cite{DHS} for the details).

 An effective algebraic criterion for the birationality of $\mathfrak{F}$ onto $Y$ was established in \cite[Theorem 2.18]{DHS}. The appropriate tool to deal with this question is the graph of $\mathfrak{F}$ which, algebraically, is thought of as the Rees algebra of a given base ideal $(\ff)\subset A=S(X)$ -- this idea goes back at least to \cite{RS}.
 For the sake of the reader's convenience and handy reference, we state the main features of this criterion.
 
 Let $\mathfrak{F}:X\dasharrow  \mathbb{P}_k^m$ denote a rational map with image $Y\subset \mathbb{P}_k^m$.
 Denote $A=k[\XX]/\fa$ and $B=k[\YY]/\fb$  the respective homogeneous coordinate rings of $X$ and $Y$. To simplify the statement we assume that $X$ and $Y$ are non-degenerated, i.e., neither $\fa$ nor $\fb$ contains nonzero linear forms. Let $\ff\subset A$ be a set of forms of fixed degree giving a representative of $\mathfrak{F}$. Let $A[\YY]=A[Y_0,\ldots,Y_m]$ be a polynomial ring over $A$ with the standard bigrading, i.e., where $\deg(x_i)=(1,0)$ and $\deg(Y_j)=(0,1)$.
 Since $\ff$ is generated in a fixed degree, the Rees algebra
 $${\mathcal R}_A((\ff)):=A\oplus I\oplus I^2\oplus \cdots \simeq A[It]\subset A[t]$$
 is a standard bigraded $k$-algebra. Mapping $Y_j\mapsto f_jt$ yields
 a presentation $A[\YY]/\mathcal{J}\simeq {\mathcal R}_A((\ff))$, with $\mathcal{J}=\bigoplus \mathcal{J}_{a,b}$ a bihomogeneous ideal.
Pick a minimal set of generators of the ideal $({\mathcal J}_{1,*})$ consisting of a finite number of forms of bidegree $(1,q)$, for various $q$'s.
Let $\{P_1,\ldots,P_s\}\subset k[\XX,\YY]$ denote liftings of these biforms. Consider the Jacobian matrix of the polynomials 
$\{ P_1,\ldots,P_s\}$ with respect to $\XX$, a matrix with entries in $k[\YY]$.
Write $\psi$ for the corresponding matrix over $B=k[\YY]/{\mathfrak b}$.
 \begin{thm}\label{Tcriterion}{\rm (}\cite[Theorem 2.18]{DHS}{\rm )} With the above notations, $\mathfrak{F}$ is birational if and only if rank $(\psi)=n$.
 	If this is the case then
 	\begin{enumerate}
 		\item[{\rm (i)}] A  representative of the inverse map over $B$ is given by the coordinates of any homogeneous vector of positive degree
 		in the {\rm (}rank one{\rm )} null space of $\psi$ over $B$.
 		
 		\item[{\rm (ii)}] A  representative as in {\rm (i)}
 		can be taken to be the ordered set of the signed $n$-minors
 		of an arbitrary $n\times (n+1)$ submatrix of $\psi$ of rank $n$.
 	\end{enumerate}
 \end{thm}
 
 Now, drawing upon this criterion and a recent bound established in \cite{MR}, one obtains the following general bound for the  degree of a representative of the inverse map to a birational map:
 
 \begin{prop}\label{Pinversedegree}   Let $X\subset \mathbb{P}_k^n$ denote a closed reduced and irreducible subvariety, and let $\mathfrak{F}:X\dasharrow  \mathbb{P}_k^m$ denote a birational map with image $Y\subset \mathbb{P}_k^m$, admitting a representative of degree $d\geq 1$.
 Assume that $X$ and $Y$ are non-degenerated.
 Then the degree of any  representative of the inverse map  $\mathfrak{F}^{-1}$ is bounded above by the number
 	$$ 2n\bigg[\frac{1}{2}\delta^{2(n+m+1-\Dim(X))^2}+\delta\bigg]^{2^{(\Dim(X)+2)}}$$
 	where $\delta=\max\{d+1,d_0\}$ and  $d_0$ is the maximum degree of a set of minimal generators of the homogeneous defining ideal of $X$ in the given embedding {\rm (}with the convention that $d_0=0$ if $X=\mathbb{P}_k^n${\rm )}.
 \end{prop}
 \begin{proof} Keeping the previously established notation, let $A=k[\XX]/\fa$  be the coordinate ring of $X$, let $I:=(\ff)\subset A$ denote the ideal generated by the forms of the given representative of  $\mathfrak{F}$, and let  $ {\mathcal R}_A(I)\simeq A[\YY]/\mathcal{J}$ be the corresponding Rees algebra. 
 According to Theorem \ref{Tcriterion} a representative for $\mathfrak{F}^{-1}$ is given by certain $n\times n$ minors of $\psi$. Since the entries  of $\psi$ are of degree at most $\deg_y({\mathcal J}_{1,*})$, it follows that the degree of any representative of $\mathfrak{F}^{-1}$ is bounded by $n\cdot{\rm reltype}_R(I)$, where ${\rm reltype}_R(I)$ denotes  the relation type of $I$ in $R$ (this is independent of $\mathcal{J}$). 
 	
To compute  $\mathcal{ J}\subset k[\XX,\YY]$, one takes a new variable $t$ over $k[\XX,\YY]$. Then  $\mathcal{J}=(Y_0-f_0t,\cdots,Y_m-f_mt, \fa )\cap k[\XX,\YY]$. To obtain explicit generators of this elimination ideal one considers an elimination term order such as $t>\XX>\YY$ and computes the Grobner basis of the ideal $J:=(Y_0-f_0t,\cdots,Y_m-f_mt, \fa )$ with respect to this order. 
 	 	
 Notice that $t$ is a non-zero divisor on $J$ and that $k[\XX,\YY,t]/(J, t)\simeq A$, hence 
 	$$\Dim(k[\XX,\YY,t]/J)=\Dim(A)+1=\Dim(X)+2.$$ 
 	
We now apply a theorem of Mayer and Ritscher \cite{MR} saying that for an arbitrary  ideal of dimension $r$ in a polynomial ring  with $l$ variables over an infinite  field, generated in total degrees  $d_1\ge d_2\ge\cdots\ge d_s$,  the degrees of the elements of a reduced Grobner basis are bounded by
 	 $$2\big [\frac{1}{2}((d_1\ldots d_{l-r})^{2(l-r)}+d_1)\big]^{2^{r}}.$$
 	 
 We apply this result to the above ideal $J\subset k[\XX,\YY]$, noting that $l=n+m+3$ and $r=\Dim(X)+2$ in the present situation.	
 \end{proof}
 
\medskip

 	The above bound is unfortunately very coarse and far from even recovering the optimal bound in the case of Cremona maps.
 	It serves a purpose of motivating the question as to whether, given closed subvarieties $X\subset \mathbb{P}^n$ and  $Y\subset \mathbb{P}^m$, there exists a function $\Phi(d,d_0,e_0,\dim(X),n,m)$ bounding the degree of any representative of the inverse map to any birational map from $X$ onto $Y$ having a representative of degree $d$. Here, $d_0$  (respectively, $e_0$) denotes the maximum degree of a minimal generator of the homogeneous defining ideal of $X$ (respectively, $Y$).
 	 Actually, since very little is known about this question, one may as well ask whether $\Phi$ is polynomial as a function of $d$.

A much better bound is obtained under the requirement that the graph of the given birational map is arithmetically Cohen--Macaulay (Theorem~\ref{Tbirational}).


\section{Upper bounds when the graph is Cohen--Macaulay}

\subsection{Plane Cremona case}

Although our main concern is the use of homological tools to understand the impact of the Cohen--Macaulay property, in a special case one can get around with  more elementary methods and in addition obtain a complete classification of plane Cremona maps with Cohen-Macaulay graph.

We recall that a Cremona map of the projective space $\pp_k^n$ is a birational map of this space onto itself.
In the case where $n=2$ a Cremona map is informally called a plane Cremona map. 
The base points of a plane Cremona map are the points of $\pp_k^2$ where the base ideal vanishes plus some infinitely near points defined in terms of a blowing-up procedure (see \cite{alberich}) -- those belonging to $\pp_k^2$ are called proper base points.

A particular, but very important, example of a plane Cremona map is a so-called de Jonqui\`eres map. Such maps constitute a fundamental building block of statements on the structure of the plane Cremona group.
They can be defined in several different ways, both geometrically and algebraically (see, e.g., \cite{alberich}, \cite{PanStellar}, \cite{HS}).

We have the following result:
\begin{prop}
Let $I\subset A:=k[x,y,z]$ denote the base ideal of a plane Cremona map defined by forms of degree $d\geq 1$ and admitting at least three proper non-aligned base points.
The following two conditions are equivalent:
\begin{enumerate}
\item[{\rm (a)}] $I$ is a saturated ideal and the Rees algebra $\mathcal{R}:=\mathcal{R}_A(I)$ of $I$ is Cohen--Macaulay.
\item[{\rm (b)}] $d\leq 3$ or else $d=4$ and the map is not a de Jonqui\`eres map.
\end{enumerate}
\end{prop}
\begin{proof}
(a) $\Rightarrow$ (b)
First, since $\mathcal{R}$ has codimension $2$, the Cohen--Macaulayness condition implies that the defining ideal $\mathcal{J}$ of $\mathcal{R}$ on the polynomial ring $A[t,u,v]$ ($t,u,v$ variables for the target projective space $\pp^2$) is a codimension two perfect ideal; as such it is generated by the maximal minors of an $(m+1)\times m$ matrix $\Psi$ with bigraded entries.

Among the minimal generators of $\mathcal{J}$ are included the ones of bidegree $(\_\,,1)$ coming from a complete set of minimal syzygies of $I$. Since we are assuming that $I$ is saturated, then $I$ is generated by the $2$-minors of a rank two $3\times 2$ matrix $\phi$. Letting $r,r'$ denote the (standard) degrees of the two columns of  $\phi$ (hence, $d=r+r'$), the ideal  $\mathcal{J}$ has two minimal generators of bidegrees $(r,1)$ and $(r',1)$.
As these are maximal minors $\Delta$ and $\Delta'$ of $\Psi$, we may assume that they are, respectively, the minor omitting the last row and the minor omitting the first row.
By suitably permuting columns one may assume that
 $\Psi$ has the form
\begin{equation}\label{HB}
\Psi=\left(\begin{array}{ccccc}
a_1 &  &  & &\\
 & a_2 &  & b_1 & \\
 & b_2 & a_3 & &\\
 &  & \ddots & \ddots &\\
&  &  & b_{m-1} & a_m\\
&  & b_m & &
\end{array}\right).
\end{equation}
where $\Delta=\overbrace{a_1a_2\cdots a_m}^{\neq 0}+\cdots$ and, similarly, $\Delta'=\overbrace{b_1b_2\cdots b_m}^{\neq 0}+\cdots$ (note the random positioning of the $b_i$'s as one may not be able to similarly dispose them  along a subdiagonal without disrupting the diagonal arrangement of the $a_i$'s, but this will turn out to be irrelevant).

Setting  ${\rm bideg}(a_i):=(r_i,s_i)$, $1\leq i\leq m$, and similarly, ${\rm bideg}(b_i):=(r'_i,s'_i)$, $1\leq i\leq m$, a minute thought, keeping in mind that $\Delta$ (respectively, $\Delta'$) has bidegree $(r,1)$ (respectively, $(r',1)$), yields $r=\sum_i r_i$ and $s_2=\cdots =s_m=0$ (respectively, $r'=\sum_i r'_i$ and $s'_1=\cdots =s'_{m-1}=0$).
For further visualization, one can draw a corresponding matrix whose entries are the bidegrees of the entries of $\Psi$:
\begin{equation}\label{bidegs}
{\rm bidegs}(\Psi)=\left(\begin{array}{ccccc}
(r_1,1) &  &  & &\\
 & (r_2,0) &  & & (r'_1,0)\\
 & (r'_2,0) & (r_3,0) & &\\
 &  & \ddots & \ddots &\\
&  &  & (r'_{m-1},0) & (r_m,0)\\
&  & (r'_m,1)  & &
\end{array}\right).
\end{equation}
Moreover, since the total degree is fixed along any row, we have $r'_i=r_{i+1}$ for $1\leq i\leq m-1$.
Therefore, one also has $r-r'=r_1-r'_m$.

Now, any other minimal generator of $\mathcal{J}$ being still a maximal minor thereof has to include the first and last rows of $\Psi$. This leaves no choice for the bidegrees of the remaining minimal generators of $\mathcal{J}$ except to have the form $(t_j,2)$, for some $t_j\geq 2$, $2\leq j\leq m$.
Moreover, necessarily, $t_j=r_1+r_2+\cdots +\widehat{r_j}+\cdots r_m+r'_m$.

On the other hand, since $I$ is the base ideal of a Cremona map,  then $\mathcal{J}$ admits at least two minimal generators of bidegree $(1,s)$, for some $s\geq 1$ (see Theorem \ref{Tcriterion}).

Say, $r'\leq r$. If $r'=1$, and since one is assuming that $I$ is saturated, then one can show that the map is a de Jonqui\`eres map -- a full proof is given in  \cite[Proposition 3.4 ]{RaSi})), but the main line of argument is as follows: if ${\rm cod} (I_1(\phi))=3$ then $I$ is a generically a complete intersection and, hence of linear type since it is an almost complete intersection; but this is not possible by \cite[Proposition 3.4]{DHS}. Thus, $I_1(\phi)$ has codimension $\leq 2$ and up to a  change of variables one may assume that the first column of $\phi$ has coordinates $x,y,0$.
Since the map is assumed to a Cremona map, the coordinates of the second column can be taken to be forms of degree at most $1$ in $z$ (see, e.g., \cite[Corollaire 2.3]{PanStellar}).

Now, by the ``only if'' part of \cite[Theorem 2.7 (iii)]{HS}, one has $d\leq 3$.

Thus, assume that $r'\geq 2$.
Then there are two distinct minimal generators of bidegree $(1,2)$.
Letting $\Delta_1$ be one of them, we must have
 $1=t_j=r_1+r_2+\cdots +\widehat{r_j}+\cdots r_m+r'_m$, for for some $2\leq j\leq m$.
It follows that $r=(r_1+r_2+\cdots +\widehat{r_j}+\cdots r_m)+r_j\leq r_j+1$.
Since the row where the entry $a_j$ appears is omitted by the minor $\Delta_1$, it must appear in the other minor $\Delta_2$ of bidegree $(1,2)$.
By a similar token, $1=t_{j'}$, with $j'\neq j$, and hence we must have $r_j\leq 1$.
We conclude that $r\leq r_j+1\leq 2$.
Since we are assuming that $2\leq r'\leq r$, it follows finally that $r=r'=2$, in which case the map has degree $4$ and is necessarily non de Jonqui\`eres.

\medskip

(b) $\Rightarrow$ (a)
This implication is mostly picking up from various results in the literature.

Assume first that $d=2$, i.e., the map is a quadratic Cremona map. This case is sufficiently well-known: the ideal is generated by the $2$-minors of a rank two $2\times 3$ matrix with linear entries. Therefore, $A/I$ is a Cohen--Macaulay ring -- equivalently, in this situation, $I$ is saturated -- and, moreover,  $I$ is an ideal of linear type, i.e., $\mathcal{R}$ coincides with the symmetric algebra of $I$,  the latter being a complete intersection.

If $d=3$ the map is a de Jonqui\`eres map as a consequence of the well-known equations of condition satisfied by a Cremona map. In this case, saturation follows, e.g., from \cite[Corollary 2.5]{HS}  and Cohen--Macaulayness follows from the ``if'' part of \cite[Theorem 2.7 (iii)]{HS}.

Finally, assume that $d=4$.  By \cite[Theorem 1.5 (i) ]{HS}, the base ideal is saturated. Since the map is not a de Jonqui\`eres map
 there is only one additional proper homaloidal type for this degree, namely, $(4;2^3,1^3)$ (this follows easily from the Hudson test as stated in \cite{alberich}). 
Since we are assuming that three of the base points are proper and non-aligned then the result is proved in \cite[Proposition A.1 and Remark A.3]{CoRaSi}.
\end{proof}

\begin{rem}\rm
If $I$ is not saturated, then to proceed elementarily as in the above argument would require knowing the degrees of a set of generating minimal syzygies of $I$. There is a strong suspicion about how these ought to be, while there is at least one large class  matching this expectation (\cite[Proposition 4.12]{RaSi}).
 Note that, for an arbitrary ideal $I$ the hypothesis that $\mathcal{R}$ is Cohen--Macaulay does not imply that $I$ is saturated.
For example, this is the case of an almost complete intersection $I\subset R$ in a Cohen--Macaulay ring, such that $I$ is a non-saturated generically complete intersection.
It is known -- see, e.g., \cite{Trento} --, that $I$ is of linear type (in particular, has maximal analytic spread) and its symmetric algebra is Cohen--Macaulay.
(Perhaps the simplest explicit example has $R=k[x,y,z]$ and $I$ the ideal generated by the partial derivatives of the binomial $y^2z-x^3$.)

Observe, however, that although the base ideal of a plane Cremona map is an almost complete intersection of maximal analytic spread, it is very rarely generically a complete intersection (\cite[Corollary 3.6]{DHS}).
And in fact, we will prove in Proposition~\ref{CCremona} that, for the base ideal of a plane Cremona map, the hypothesis that $\mathcal{R}$ is Cohen--Macaulay {\em does} imply that $I$ is saturated.
This will allow us to remove the extra assumption in item (a) of the above proposition.
\end{rem}

\subsection{General case: preliminaries}


From previous parts, given a rational map $\mathfrak{F}:X\dasharrow \mathbb{P}^m_k$, the Rees algebra $\mathcal{R}_A(I)$ of one of its base ideals $I\subset A$ depends only on the map and not on this particular base ideal.
Since the biprojective variety defined by $\mathcal{R}_A(I)$ is the graph of $\mathfrak{F}$, if no confusion arises any property of the Rees algebra will be said to be a property of the latter.

Let us denote by $\delta(\mathfrak{F})$  the least possible degree of the forms of a base ideal of the map.
For the next proposition we introduce a new numerical invariant of a rational map $\mathfrak{F}:X\subset \mathbb{P}^n_k \dasharrow  \mathbb{P}^m_k$.
It will be a consequence of the following lemma.

\begin{lem}
Let $I,J\subset A$ denote two base ideals of a rational map $\mathfrak{F}:X\subset \mathbb{P}^n_k \dasharrow  \mathbb{P}^m_k$, where $A$ stands for the homogeneous coordinate ring of $X$. Let $\delta_I, \delta_J$ denote the respective degrees of the generators.
Then
$$\Reg(I^r)-r\delta_I=\Reg(J^r)-r\delta_J,$$
for every $r\geq 1$.
\end{lem}
\begin{proof}
As observed in Section~\ref{Birational_Maps}, $I$ and $J$ have the same syzygies.
This is because being base ideals of the same rational map they come from equivalent representatives in the sense of \cite[Section 2.2]{DHS} -- that is, two sets of homogeneous coordinates of the same point of $ \mathbb{P}^n_{K(A)}$, where $K(R)$ is the total ring of fractions of $A$.
It follows that, for any $r\geq 1$, the powers $I^r$ and $J^r$ are base ideals of another rational map, and hence by the same token they too have the same syzygies.
Writing $A$ as a residue of the polynomial ring $R=k[x_0,\ldots,x_n]$, let
\begin{equation}
\label{resI}
\cdots \lar \oplus_{i}R(-r\,\delta_I-i) \lar R(-r\,\delta_I)^{N_r}\lar I^r\rar 0
\end{equation}
denote the minimal free resolution of $I^r$ over $R$, for suitable $N_r$.
Then the above discussion implies that the minimal free resolution of $J^r$ over $R$ has the form
$$\cdots \lar \oplus_{i}R(-r\,\delta_J-i) \lar R(-r\,\delta_J)^{N_r}\lar J^r\rar 0$$
where the maps and the standard degrees $i$'s are the same as those of (\ref{resI}).

The stated equality follows  immediately from this.
\end{proof}

Clearly, $\Reg(I^r)\geq r\delta_I$  for every $r\geq 1$.
Thus, one gets a numerical function $\mathfrak{f}:\mathbb{N}\rar \mathbb{N}$ which depends solely on the map $\mathfrak{F}$, defined by $\mathfrak{f}(r)=\Reg(I^r)- r\delta_I$, for any choice of a base ideal $I$.
This function attains a maximum which, according to Chardin-Romer (\cite[Theorem 3.5]{Ch}), is the $\xx$-regularity $\Reg_{\xx}(\mathcal{R}_{A}(I))$ introduced in Section~\ref{homological_recap}, thus retrieving another numerical invariant of $\mathfrak{F}$ (since as remarked in Section~\ref{Birational_Maps}, the Rees algebra is independent of the chosen base ideal).

As a parenthesis, a question arises naturally in the birational case, as to what values of $r$ give the maximum of $\mathfrak{f}$ defined by $\mathfrak{f}(r)=\Reg(I^r)- r\delta_I$, for any choice of a base ideal $I$.
It is conceivable that, for arbitrary homogeneous ideals generated in a fixed degree, the behavior is quite erratic.
However, for a base ideal $I$ of a birational map, one knows, for example, that far out powers of $I$ will have quite a bit of linear syzygies, while the regularity is quite sensitive to those.

And, in fact, the maximum may fail to be attained at $r=1$, as is shown by the following example due to  Terai  (stated in \cite[Remark 3]{Conca} and subsequently improved in \cite[Theorem 1.1]{St}):
\begin{example}\label{Ex1}\rm  Let
	$$I=(\mathbf{f}):=(abc,abf,ace,ade,adf,bcd,bde,bef,cdf,cef)\subset A:=\mathbb{Q}[a,b,c,d,e,f].
	$$
\end{example}
The main highlights of this example are:
\begin{enumerate}
	\item[{\rm (1)}] $I$ has linear resolution -- in particular, $\Reg(I)=3$.
	\item[{\rm (2)}] $I^2$ has codepth zero and the last syzygies have standard degree $2$, hence $\Reg(I^2)\geq 12-5=7$ (in fact, one has $=7$).
	\item[{\rm (3)}] $I$ defines a birational map onto the image.
\end{enumerate}
The first two features are directly checked in a computer program.
To see (3), it suffices to know that $I$ is linearly presented (or even still that the linear part of the syzygies has maximal rank, i.e., $9$) and that $\dim[\mathbf{f}]=6$ (maximum).
These two properties imply birationality by \cite[Theorem 3.2]{DHS}.
For the dimension, since we are in characteristic zero, it suffices to check that the Jacobian matrix of $\mathbf{f}$ has rank $6$.

The above question gets streamlined in the particular case  of Cremona maps. Here it might be tempting to guess that the maximum is attained at $r=1$, but at the moment we lack sufficient preliminary evidence.

\subsection{General case: main theorems}

In the case where $\mathfrak{F}$ is birational onto its image, one can state a bound for the minimum degree $\delta(\mathfrak{F})$ of a representative:

\begin{thm}\label{preg}\label{Tbirational}
Let $X\subset \mathbb{P}^n_k$ denote a closed  reduced and irreducible nondegenerate subvariety and let $\mathfrak{F}:X\dasharrow  \mathbb{P}^m_k$ denote a birational map with nondegenerate image $Y$. Let there be given an arbitrary representative of $\mathfrak{F}$, $\delta$ its degree, and $I$ the base ideal generated by the forms of this representative. Then
\begin{equation}\label{bound2}
\delta(\mathfrak{F}) \leq m\,\sup_{r\geq 1}\{\Reg(I^r)-r\delta+1\}.
\end{equation}
If moreover the graph of $\mathfrak{F}$ is Cohen-Macaulay then
\begin{equation}\label{bound21}
\delta(\mathfrak{F}) \leq m\,\sup_{r\geq 1}\{\Reg(I^r)-r\delta+1\} \leq m\,d
\end{equation}
and, similarly,
\begin{equation}\label{bound22}
\delta(\mathfrak{F^{-1}})  \leq n\,d,
\end{equation}
  where $d=\Dim(X)+1$. 
\end{thm}
\begin{proof}
Let $A$ and $B$ stand, respectively,  for the homogeneous coordinate rings of $X$ and $Y$.
Let $I'\subset B$ denote the ideal generated by the forms of a representative of the inverse map to $\mathfrak{F}$.
Consider a presentation  $\mathcal{R}_{B}(I')\simeq B[x_0,\cdots,x_n]/\mathcal{J}_{I'}$, where $B$ is concentrated in degree zero.  According to Theorem \ref{Tcriterion}, the map $\mathfrak{F}$, considered as the inverse map to its own inverse, has a representative whose coordinates are suitable $m\times m$ minors of the Jacobian  matrix ${\rm Jacob}_{\yy}((\mathcal{J}_{I'})_{(\ast,1)})$ and, clearly, the degree of these coordinates is an upper bound for the minimum degree $\delta(\mathfrak{F})$ of a representative.
Therefore,  to bound $\delta(\mathfrak{F})$ one is led to bounding from above the $\xx$-degree of the minimal generators of $\mathcal{J}_{I'}$, i.e. the relation type of the ideal $I'\subset B$.

Now, quite generally, the relation type of the ideal $I'\subset B$ is bounded by $\Reg(\mathcal{R}_{B}(I'))+1$, where the regularity is computed with respect to the the irrelevant ideal $\mathcal{R}_{B}(I')_+$ (\cite[15.3.1]{BS}).

We claim that $\Reg_{\xx}(\mathcal{R}_{A}(I)) =\Reg(\mathcal{R}_{B}(I'))$.

For this, we note that birationality implies an  isomorphism  $\mathcal{R}_{B}(I')\simeq \mathcal{R}_{A}(I)$ of bigraded $k$-algebras (\cite[Proposition 2.1]{S}) naturally based on the standard bigraded structure of the polynomial ring $R[\yy]=k[\xx,\yy]$.
If one changes this grading by setting  $\deg(\yy)=0$ and $\deg(\xx)=1$, it obtains an isomorphism of standard graded $k$-algebras.
It follows that $H^i_{(\xx)}(\mathcal{R}_{B}(I'))\simeq H^i_{(\xx)}((\mathcal{R}_{A}(I))$ for every $i$.
Since $\Reg(\mathcal{R}_{B}(I'))+1$ is computed with respect to the the irrelevant ideal $\mathcal{R}_{B}(I')_+$, it is
 defined in terms of local cohomology based on the ideal $B[\xx]_+=(\xx)$.
 Then the above  isomorphisms of local cohomology modules impliy that $\Reg_{\xx}(\mathcal{R}_{A}(I)) =\Reg(\mathcal{R}_{B}(I'))$.

As a consequence,   $\delta(\mathfrak{F})\leq m(\Reg_{\xx}(\mathcal{R}_{A}(I))+1)$. Applying Chardin-Romer 's equality mentioned earlier, we obtain the first inequality.

If, in addition, the graph of $\mathfrak{F}$ is Cohen-Macaulay, then the isomorphism  $\mathcal{R}_{B}(I')\simeq \mathcal{R}_{A}(I)$ of bigraded structures implies that $\mathcal{R}_{B}(I')$ is Cohen-Macaulay. By \cite[Proposition 6.2]{AHT}, one has $\Reg(\mathcal{R}_{B}(I'))\leq \ell(I')-1$. Since  $\ell(I')=\dim(X)+1$ we are through for the second bound.

The third bound is a consequence of the symmetry of the argument and the fact that both $\mathfrak{F}$ and its inverse share the same graph.
\end{proof}

We note that the bounds just obtained are expressed solely in terms of the ambient dimension, whereas the Cohen-Macaulayness of the graph is a property of the map.

For the subsequent result we will use the following lemma -- essentially a collection of known results in a form that suits our purpose in this part.

\begin{lem}\label{lregR}Let $(A,\fm)$ be a Cohen--Macaulay local ring of dimension $d$ and let $I$ be an ideal of $A$ with analytic spread $\ell$ and reduction number $r(I)$. If $\mathcal{R}_A(I)$ is Cohen--Macaulay, one has
$$
\Reg(\mathcal{R}_A(I))\leq \left \{
\begin{array}{cc}
r(I)& {\text if} \,\ell=\Ht(I) \\
\max\{r(I),\ell-2\}& {\text otherwise}.
\end{array}
\right.$$
\end{lem}
\begin{proof}
Set $\mathcal{R}:=\mathcal{R}_A(I)$ and $\mathcal{G}:=\mathcal{G}_I(A)$. Since both $A$ and $\mathcal{R}$ are Cohen-Macaulay then $\mathcal{G}$ is Cohen-Macaulay and moreover $a^*_d(\mathcal{G})\leq -1$ \cite[Theorem 1.1]{TI}. According to   \cite[Corollary 4.2]{HHK}, one has   $$a_d^*(\mathcal{G})=\max\{r(I)-\ell(I),a(\mathcal{G}_{\fp})\,|\, \fp\in V(I)\setminus {\fm}\}$$
where  $a(\mathcal{G}_{\fp})$ is the $a$-invariant of $\mathcal{G}_{\fp}$ over $A_{\fp}$.
A recursion then yields
\begin{equation}\label{eHHK}
a_d^*(\mathcal{G})=\max\{r(I_{\fp})-\ell(I_{\fp})\,|\, \fp\in V(I)\}.
\end{equation}

On the other hand, by \cite[Corollary 4.11]{AHT}, $\Reg(\mathcal{G})\leq \max\{r_{\ell-1}(I),r(I)\}$, where $$
 r_{\ell-1}(I)=\left \{
\begin{array}{ll}
-1, & {\rm if} \;\ell=\Ht(I) \\
\max\{r(I_{\fp})-\ell(I_{\fp}): \Ht(\fp)\leq \ell-1 \text{ and  }\ell(I_{\fp})= \Ht(\fp) \}+\ell-1, & \text{ otherwise.}
\end{array}
\right.$$

Since  $a^*_d(\mathcal{G})\leq -1$, (\ref{eHHK}) implies that $r(I_{\fp})-\ell(I_{\fp})\leq -1$ for all $\fp\in V(I)$. Therefore $r_{\ell-1}(I)\leq \ell-2$ if $\ell\not=\Ht(I)$. The assertion now follows  because $\Reg(\mathcal{R})=\Reg(\mathcal{G})$ by \cite[Corollary 3.3]{T}.
\end{proof}

\begin{thm}\label{CCremona} {\rm (}$k$ infinite{\rm )} Let   $\mathfrak{F}:\mathbb{P}^n_k\dashrightarrow\mathbb{P}^n_k$  be a Cremona map having Cohen--Macaulay graph.
Then the respective base ideals of $\mathfrak{F}$ and its inverse are generated in degree $\leq n^2$.

Furthermore, if $I\subset A=k[x_0,\cdots,x_n]$ stands for the base ideal of $\mathfrak{F}$, then the saturation of $I$ is the ideal $(I:_A A_+ ^{n-2})$.
In particular, the base ideal of a plane Cremona map with Cohen--Macaulay graph is saturated.
\end{thm}
\begin{proof}
Set $I\subset A$ for the base ideal of the map.
 Since the generators of $I$ are analytically independent and $k$ is infinite, the reduction number $r(I)=0$. From Lemma~\ref{lregR} one has $\Reg(\mathcal{R}_A(I))\leq n-1$, hence the relation type of $I$ is $\leq n$. Therefore the same reasoning as in the proof of Theorem \ref{Tbirational} shows that $\delta(\mathfrak{F})$ and $\delta(\mathfrak{F}^{-1})$  are at most $n^2$.

As earlier, twisting around by considering the graph of the inverse map, we derive the inequality $\Reg_{\xx}(\mathcal{R}_{A}(I))\leq n-1$.
It then follows from Chardin-Romer's equality that  $\Reg(I)-\delta\leq n-1$ where $\delta$ is the degree of the generators of $I$. Therefore  $\Reg(A/I)\leq \delta+n-2$. In particular $\End(H^0_{(\xx)}(A/I))=\End(I^{sat}/I)\leq \delta+n-2$. On the other hand, by \cite[Proposition 1.2 (i)]{PR} (also  \cite[Proposition 1.3.6]{D}) one always has $\beg(I^{sat}/I)\geq \delta+1$ whenever $I$ is the base ideal of a Cremona map. Therefore $(A_+) ^{n-2}\frac{I^{sat}}{I}=0$ which proves the assertion.
\end{proof}

\subsection{Confronting a na\"ive upper bound} 

In this short piece we deviate somewhat from our main objectives by  giving an upper bound for the degree of a linear system defining a rational map, under the assumption  that it generates an ideal of grade at least $2$ -- an assumption meaning that the map has essentially a unique base ideal.

\begin{prop}\label{grade2}
Let $\mathfrak{F}: X\subset \pp^n_k\dasharrow \mathbb{P}^m_k$ denote a rational map. Suppose that $\mathfrak{F}$ admits a base ideal $I\subset A$ of grade at least $2$. If $\delta$ denote the degree of the forms generating $I$, one has
$$\delta\leq \frac{m}{m+1}\, e,
$$
where $e$ is the largest twist in the graded presentation of $I$ over $A$.
\end{prop}

The proposition is an immediate consequence of the following general lemma which seems to be folklore.

\begin{lem}\label{grade2_nonsense}
Let $R$ be a commutative Noetherian ring and $I\subset R$ be an ideal of grade at least two. Fixing a set $\{f_0,\ldots ,f_m\}$ of generators of $I$, let $M$ denote a corresponding presentation matrix. Then there exists an $(m+1)\times m$ submatrix of $M$ of rank $m$,  and a non-zero element $h\in R$ such that $hf_i=\Delta_i$, where $\Delta_i$ denotes the $i$-th minor of this submatrix.

In particular, if $R$ is standard graded and if $I$ is homogenous and generated in a single degree $\delta$ and related in degree at most $b_1(I)$ then $b_1(I)\geq \frac{m+1}{m}\,\delta$.
\end{lem}
\begin{proof} Consider the free presentation $R^p\xrightarrow{M}R^{m+1}\xrightarrow{\mathbf{f}}R\to R/I\to0$, where $\mathbf{f}$ denotes the vector $(f_0\cdots f_m)$.
Since $\grade(I)\geq 2$, dualizing into $R$ yields the exact sequence $0\to R^*\xrightarrow{\mathbf{f}^t}R^{{m+1}^*}\xrightarrow{M^t}R^{p^*}$ (as is well-known, the natural inclusion $R\hookrightarrow {\rm Hom}(I,R)$ is an equality if and only if $I$ has grade at least $2$) and moreover   $\Image(M)$ has rank $m$.
Let $N$ be an $m\times (m+1)$ submatrix of $M^t$ of rank $m$ and denote the (ordered) $m\times m$ minors of $N$ by $\Delta=\Delta_0,\ldots,\Delta_m$. 
For an arbitrary row of $M^t$, say, $(r_{0i}\cdots r_{mi})$, $\sum_{j=0}^{j=m}r_{ji}\Delta_j$ is the determinant of an $(m+1)\times(m+1)$ matrix which is either a matrix having two equal rows or an $(m+1)\times(m+1)$ submatrix of $M^t$. In any case the corresponding determinant vanishes. Therefore in the following diagram the compositions of any two compatible maps are zero
$$
 \xymatrix{
   &  R^* \ar^{\mathbf{f}^t}[r]   &  R^{{n+1}^*}\ar^{M^t}[r]   &  R^{p^*}\\
   &                       &  R\ar_{(\Delta)}[u]        &\\
  }.
 $$
Since $\ker(M^t)=\Image\mathbf{f}^t$, one has $\Delta=(\Delta_0\cdots\Delta_n)\in \Image(\mathbf{f}^t)$; which proves the first assertion.
To see the degree assertion in the equihomogeneous case, note that in the graded case $h$ is homogeneous, hence the equality $hf_i=\Delta_i$ implies that $\delta \leq \Deg(\Delta_i)\leq m\,(b_1(I)-\delta)$.
\end{proof}

Since the bound in Proposition~\ref{grade2} is quite natural, it makes sense confronting it with the birational bound obtained in Theorem~\ref{preg}.
For this we assume that the map $\mathfrak{F}$ is birational, as in Theorem~\ref{preg}, and $I\subset A$ is a base ideal of grade $\geq 2$, as in Proposition~\ref{grade2}.
In this situation, the representative of $\mathfrak{F}$ is uniquely defined up to scalars (see \cite[Proposition 1.1]{S}), hence $\delta(\mathfrak{F})=\delta$.
Let us suppose in addition that the function $\mathfrak{f}$ attains its maximum at $r=1$ -- a hypothesis we will digress on in a minute.

Then the second bound gives $\delta\leq m (\Reg(I)-\delta+1)$, i.e., $\delta\leq \frac{m}{m+1} (\Reg(I)+1)$.
Since $\Reg(I)\geq e-1$, we see that the first bound is better.

On the other hand, the second bound is more comprehensive and besides it has the advantage of  providing a lower bound for $\mathfrak{f}$ in terms of an invariant of the map $\mathfrak{F}$ and not of a particular base ideal $I$.

{\bf Acknowledgments.} Both authors were partially
supported  by the Math-Amsud project 15MATH03- SYRAM. The second author was additionally partially supported by a grant from CNPQ (302298/2014-2) and a CAPES Senior Visiting Researcher grant at Universidade Federal da Para\ii ba  (5742201241/2016).
Part of this work was done while the first author   was  supported by a CNPq grant (300586/2012-4) and by a ``p\'os-doutorado no exterior'' 233035/2014-1 for which he is thankful.  He would also like to thank  the University of Utah for the hospitality during the period. Both authors  thank J\'er\'emy Blanc for calling their attention to the problem of finding the degrees of the inverse map and also for mentioning the reference \cite{BBV}.



\begin{thebibliography}{AAA}
	

\bibitem[AHT]{AHT}{I. Aberbach, C. Huneke and  N.V. Trung \textit{Reduction numbers, Brian\c{c}on-Skoda theorems and the depth of Rees rings},
	Compositio Math. {\bf 97} (1995), 403--434.}
	

\bibitem[A]{alberich}{M. Alberich-Carrami\~nana, \textit {Geometry of the
		Plane Cremona Maps}, Lecture Notes in Mathematics, {\bf 1769},
	Springer-Verlag Berlin-Heidelberg,  2002.}


\bibitem[BCW]{BCW} H. Bass, E. Connell, D. Wright, \textit{The Jacobian conjecture: reduction of degree and formal expansion of the inverse}  Bull. Amer. Math. Soc. (N.S.) 7 (1982), no. 2, 287--330.


\bibitem[BBV]{BBV} F. Bogomolov, C. Bhning, H-C. Von Bothmer, \textit{Birationally isotrivial fiber spaces}  European Journal of Mathematics, March 2016, Volume 2, Issue 1, pp 45--54.



\bibitem[BS]{BS} M. Brodmann,  R.Y. Sharp, \textit{Local Cohomology, An Algebraic Introduction with Geometric Application}, Cambridge University Press, Cambridge, 1998.

\bibitem[BH]{BH}{W. Bruns and J. Herzog, {\it Cohen--Macaulay Rings},
Cambridge Studies in Advanced Mathematics, vol. 39, Cambridge
University Press, 1993.}

\bibitem[Ch]{Ch}{M. Chardin \textit{Powers of ideals: Betti numbers, cohomology and regularity},
Commutative Algebra, pp. 317--313, Springer, New York, 2013.}

\bibitem[Con]{Conca}{A. Conca, \textit{Hilbert function and resolution of the powers of the ideal of the rational normal curve}, Journal of Pure and Applied Algebra {\bf 152} (2000) 65--74.}

\bibitem[CRS]{CoRaSi}{B. Costa, Z. Ramos and A. Simis, \textit{A theorem about Cremona maps and  symbolic Rees algebras,}  International J. of Algebra and Computation, {\bf 24} (2014), 1191--1212.}

\bibitem[D]{D} I. Dolgachev, \textit{Cremona special sets of points in products of projective spaces. Complex and differential geometry,} pp. 115--–134, Springer Proc. Math., 8, Springer, Heidelberg, 2011.

\bibitem[DHS]{DHS}  A.V. Doria, S. H. Hassanzadeh , A. Simis \textit{ A characteristic free criterion of birationality ,}   Advances in Math., \textbf{230}, Issue 1, (2012), 390--413.


\bibitem[HS]{HS}{S. H. Hassanzadeh and A. Simis, \textit{ Plane Cremona maps: saturation, regularity and fat ideals}, J. Algebra, {\bf 371} (2012), 620--652.}


\bibitem[HHK]{HHK} M. Herrmann, E. Hyry and T. Korb, \textit{On $a$-invariant formulas}, J. Algebra. \textbf{227} (2000), 254--267.

\bibitem[HSV]{Trento} J. Herzog, A. Simis and W. V. Vasconcelos, \textit{ Koszul
homology and blowing-up rings}, in  {Commutative Algebra},
Proceedings: Trento 1981 (S. Greco and G. Valla, Eds.),Lecture Notes in Pure and Applied Mathematics, {Vol. 84}, Marcel Dekker, New York,
1983, pp. 79--169.

\bibitem[HR]{HR}{M. Hochster and L. J. Ratliff, Jr., \textit{Five theorems on Macaulay rings}, Pacific J. Math, {\bf 44}, (1973) 147--172.}

\bibitem[MR]{MR}{E. Mayr, S. Ritscher \textit{Dimension-dependent bounds for Grobner bases of polynomial ideals}, Journal of Symbolic Computation 49 (2013) 78--94.}

\bibitem[Pan]{PanStellar}{I. Pan, \textit{Les transformations de Cremona stellaires}, Proc. Amer. Math.
   Soc. {\bf 129} (2001),  1257--1262.}

\bibitem[PR]{PR}{I. Pan and F. Russo \textit{Cremona transformations and special double structures},
Manuscripta Math. {\bf 177} (2005), 491--510.}

\bibitem[RaS]{RaSi}{Z. Ramos and A. Simis, \textit{Homaloidal nets and ideals of fat points I}, LMS Journal of Computation and Mathematics, {\bf 19}  (2016), 54--77.}


\bibitem[RuS]{RS}{F. Russo and A. Simis, \textit{On birational maps and Jacobian matrices}, Compositio Math. {\bf 126} (2001), 335--358.}


\bibitem[S]{S}{A. Simis, \textit{Cremona transformations and some related algebras}, J. Algebra {\bf 280} (2004), 162--179.}

\bibitem[St]{St}{B. Sturmfels, \textit{Four Counterexamples in
Combinatorial Algebraic Geometry}, J. Algebra {\bf 230} (2000) 282--294.}


\bibitem[T]{T}{N. V. Trung, \textit{The Castelnuovo-Mumford regularity of the Rees algebra and the associated graded ring},
Trans. Amer. Math. Soc. {\bf 350} (1998), 2813--2832.}

\bibitem[TI]{TI}{N. V. Trung and Sh. Ikeda, \textit{When is the Rees algebra Cohen-Macaulay?},
Comm. Alg. {\bf 17} (1989), 2893--2922.}

\end{thebibliography}
\end{document}